\documentclass[12pt]{article}

\usepackage[english]{babel}
\usepackage{graphicx}
\usepackage{latexsym}
\usepackage{amsfonts}
\usepackage{amscd}
\usepackage{amssymb}
\usepackage{amsmath}
\usepackage{mathrsfs}
\usepackage{amsthm}
\usepackage{tikz}

\input xy
\xyoption{all}

\newcommand{\NN}{\mathbb{N}}

\newcommand{\vertiii}[1]{{\left\vert\kern-0.25ex\left\vert\kern-0.25ex\left\vert #1 
    \right\vert\kern-0.25ex\right\vert\kern-0.25ex\right\vert}}
\newcommand{\re}{\mathbb{R}}

\newcommand{\bz}{\mathbb{Z}}

\newcommand{\p}{\mathcal{P}}
\newcommand{\om}{\Omega}

\newcommand{\eps}{\varepsilon}

\newcommand{\K}{\mathcal{K}}

\newtheorem{theo}{Theorem}[section]
\newtheorem{lem}[theo]{Lemma}

\newtheorem{prop}[theo]{Proposition}

\theoremstyle{definition}
\newtheorem{rmk}[theo]{Remark}

\newtheorem{ex}[theo]{Example}

\begin{document}

\title{K-theory of crossed products of tiling C*-algebras by rotation groups}
\author{Charles Starling\thanks{Departamento de Matem\'atica, Universidade Federal de Santa Catarina, 88040-900 Florian\'opolis SC, Brazil. Research supported in part by NSERC (Canada) and CNPq(Brazil). Email: \texttt{slearch@gmail.com}}}
\date{}
\maketitle
\begin{abstract}
Let $\om$ be a tiling space and let $G$ be the maximal group of rotations which fixes $\om$. Then the cohomology of $\om$ and $\om/G$ are both invariants which give useful geometric information about the tilings in $\om$. The noncommutative analog of the cohomology of $\om$ is the K-theory of a C*-algebra associated to $\om$, and for translationally finite tilings of dimension 2 or less the K-theory is isomorphic to the direct sum of cohomology groups. In this paper we give a prescription for calculating the noncommutative analog of the cohomology of $\om/G$, that is, the K-theory of the crossed product of the tiling C*-algebra by $G$. We also provide a table with some calculated K-groups for many common examples, including the Penrose and pinwheel tilings.
\end{abstract}

\section{Introduction}

From a dynamical point of view, one studies an aperiodic tiling $T$ of $\re^n$ by considering a set of tilings $\om_T$ which contains $T$ and which is invariant under translations. This set is given a metric, and under some natural assumptions on $T$ the space $\om_T$ is compact and the dynamical system $(\om_T, \re^d)$ (where $\re^d$ acts on $\om$ by translation) is minimal (that is, every orbit is dense).

Invariants of both the space $\om_T$ and the dynamical system $(\om_T, \re^n)$ are prominent in the theory of tilings. The integer cohomology of the metric space $\om$ is a topological invariant which gives geometric information about $T$, and in most cases of interest is computable, see \cite{AP98}, \cite{FHK}, \cite{BDHS}. Putnam and Kellendonk \cite{KP06} have also considered the dynamical cohomology of $(\om_T, \re^n)$ and studied how this invariant is related to the cohomology of $\om_T$. 

Since $(\om_T, \re^n)$ is minimal, the quotient topology on its space of orbits is trivial. It is the view of Connes noncommutative geometry \cite{CoNCG} that in the cases of such quotients, a more suitable object to study in its place is a C*-algebra which encodes the original space and the quotient. In our case, the C*-algebra to consider is the {\em crossed product} of the C*-algebra $C(\om_T)$ by the group $\re^n$, denoted $C(\om_T)\rtimes \re^n$. An important invariant of C*-algebras is {\em K-theory}, and in many cases, the K-theory of $C(\om_T)\rtimes \re^n$ is isomorphic to the cohomology of $\om_T$, see \cite{AP98}, \cite{FHK}.

In the study of tilings, rotations also play an important r\^ ole. One of the motivating applications of aperiodic tilings is modeling quasicrystals, that is, materials which have orderly atomic structures but no translational symmetry. It was the presence of five-fold rotational symmetry in the diffraction pattern of an alloy studied by Shechtman et al \cite{Sh84} (a symmetry impossible for substances with periodic atomic structures) which began the study of these materials, and eventually earned Shechtman a Nobel prize in chemistry in 2011.

Here, we concern ourselves only with tilings of $\re^2$ which exhibit some rotational symmetry. In this setting, rotations are often a key feature and a source of aesthetic appeal. For example, the Penrose tilings famously have ten-fold rotational symmetry, and so the group of rotations generated by rotation by $\pi/5$ fixes $\om_T$, even though the smallest rotation which fixes any single tiling in $\om_T$ is $2\pi/5$. We note that the largest group of rotations which fixes $\om_T$ is not always finite -- in the case of Conway's pinwheel tilings, $\om_T$ is invariant under every rotation of the plane.

Connecting the maximal group of rotations $G$ of a given tiling $T$ with the cohomology has been undertaken in the literature by many authors, see \cite{BDHS}, \cite{ORS}, and \cite{Rand06} for example. In many cases the cohomology of the space $\om_T/G$ is computable, and is isomorphic to the ``rotationally invariant part'' of the cohomology, up to finite extensions (\cite{ORS}, Theorem 2). From the quasicrystal perspective, symmetries in diffraction spectra were considered by Mermin in \cite{Me97} and later by Lenz and Moody in \cite{LM09}.

In this note, we relate rotation groups which fix $\om_T$ to C*-algebras and K-theory. If a rotation group $G$ fixes $\om_T$, then $G$ acts in a natural way on $C(\om_T)\rtimes \re^2$, allowing us to form an iterated crossed product, $(C(\om_T)\rtimes\re^2)\rtimes G$. As one might guess, this is isomorphic to the crossed product $C(\om_T)\rtimes (\re^2 \rtimes G)$, where $\re^2 \rtimes G$ is the usual semidirect product of groups. Our goal here is to give a complete prescription for calculating the K-theory of this C*-algebra. 


What we prove (Theorem \ref{maintheorem}) is that one can calculate the K-theory of this C*-algebra if one knows how to calculate the K-theory of $C(\om/G)$ and also knows the number of the tilings with nontrivial rotational symmetry and their stabilizer subgroups. For substitution tilings, one can determine this information from rotationally invariant patches; see \cite{ORS}, \cite{BDHS}.

This note is organized as follows: in Section 2 we recall tiling terminology and set notation, while recalling the tiling C*-algebra. In Section 3 we present the main theorem, which is a prescription for calculating the K-theory of $C(\om_T)\rtimes (\re^2\rtimes G)$, and give an example which illustrates how we apply the techniques of \cite{EE11}. Section 4 is a table of some computed K-groups for frequently-studied examples, along with a remark about the consequences of some of the results to C*-algebra classificaion.
\section{Tilings and C*-algebras}

In this section we recall the definition of a tiling and its continuous hull, and list some of its properties. We also discuss the C*-algebra typically associated to a tiling.

Let $B_r(x)$ denote the open ball in $\re^2$ around $x\in\re^2$ with radius $r>0$. By a {\em tile} we mean a subset of $\re^2$ which is homeomorphic $\overline{B_1(0)}$, possibly with a label. Normally the label will be suppressed, but in all technicality a tile will be an ordered pair $(t, \text{label})$ where $t$ is a subset of $\re^2$ homeomorphic to the closed unit ball. A {\em partial tiling} is a set of tiles whose interiors have pairwise disjoint intersections. The {\em support} of a partial tiling $P$ is the union of its tiles, and is denoted supp$(P)$. A {\em patch} is a finite partial tiling, and a {\em tiling} is a partial tiling with support equal to $\re^2$. Given a partial tiling $T$ and a subset $U$ of $\re^2$, $T(U)$ will denote the partial tiling consisting of all tiles in $T$ which intersect $U$; if $x\in\re^2$, $T(\{x\})$ will be shortened to $T(x)$.

Let $E^2$ denote the group of orientation-preserving isometries of $\re^2$; this is also called the group of {\em Euclidean motions} of the plane, and can be seen as a semidirect product $E^2 = \re^2\rtimes \mathbb{T}$. Elements of $E^2$ are ordered pairs $(x, g)$ with $x\in\re^2$, $g\in \mathbb{T}$ with product and inverse determined by the formulas
\[
(x,g)^{-1} = (-g^{-1}x, g^{-1}), \hspace{1cm} (x,g)(y,h) = (x+gy, gh).
\]
In the above we consider $g\in \mathbb{T}$ as a rotation matrix, and for $x\in\re^2$, $gx$ denotes the usual matrix product. Since $\re^2$ and $\mathbb{T}$ are metric spaces, we give $\re^2\rtimes \mathbb{T}$ the metric which is the coordinatewise sum of these. For $(x,g)\in E^2$, we let $\vertiii{(x,g)}$ denote the distance from $(x,g)$ to the identity $(0,1)$.

The group $E^2$ acts on $\re^2$ on the left via the formula $(x,g)y = gy + x$. For $(x,g)\in E^2$ and $U\subset \re^2$ we denote $(x,g)U = \{ (x,g)u \mid u\in U\}$; we will make use of the shorthand $(x,1)U = U+x$ and $(0,g)U = gU$. Given a tiling $T$ and $(x,g)\in E^2$ we let $(x,g)T = \{(x,g)t \mid t\in T\}$, and we easily see that this is also a tiling. A tiling for which $T+x = T$ for some non-zero $x\in\re^2$ is called {\em periodic}. A tiling for which no such non-zero vector exists is called {\em aperiodic}.

For $r>0$, let Patch$_r(T)$ denote the set of patches $P\subset T$ such that the support of $P$ has diameter less than $r$. If $P,Q\in$ Patch$_r(T)$ and $P = (x,g)Q$ for some $(x,g)\in E^2$, we write $P\sim_{E^2} Q$. If $P = Q+x$ for some $x\in\re^2$, we write $P\sim_{\re^2}Q$. We note that both of these define equivalence relations on Patch$_r(T)$. A tiling $T$ is said to have {\em finite local complexity}, or {\em FLC}, if for each $r>0$ the set Patch$_r(T)/\sim_{E^2}$ is finite, and it is said to have {\em translational finite local complexity} if for each $r>0$ the set Patch$_r(T)/\sim_{\re^2}$ is finite. We will shorten ``$T$ has translational finite local complexity'' to ``$T$ is {\em translationally finite}''.

If $T$ is a tiling and $\p$ is a finite set of tiles such that each element of $T$ is congruent to some element of $\p$, we say that $\p$ is a set of {\em prototiles} for $T$. In the translationally finite case, we will insist that each element of $T$ be a translate of an element of $\p$.

One can define a distance between tilings under which two tilings will be close if they agree on a large ball around the origin up to a small Euclidean motion. For tilings $T_1,T_2$ let
\begin{eqnarray*}
d(T_1,T_2) & = &  \inf\{1,\varepsilon \mid \exists\ e_1, e_2\in E^2 \ni \vertiii{e_1}, \vertiii{e_1} < \eps,\\
        &   & \hspace{1.5cm}(e_1T_1)(B_{1/\eps}(0)) = (e_1T_2)(B_{1/\eps}(0))\}.
\end{eqnarray*} 
This is called the {\em tiling metric}. Given a tiling $T$, the completion of the set $T+\re^2 = \{T+x\mid x\in\re^2\}$ in this metric is denoted $\Omega_T$ and is called the {\em continuous hull} of $T$. Given a Cauchy sequence in $T+\re^2$ one can construct a tiling that it converges to (not necessarily in $T+\re^2$), and so the elements of $\Omega_T$ are all tilings. Furthermore, the metric above extends to the completion. There are many metrics in the literature which result in the same topology as this for the classes of tilings we are interested in, and one may describe this topology in a more natural way, see \cite{BHZ00}. If $T$ admits a finite prototile set and has finite local complexity, then $\Omega_T$ is compact (\cite{RW92}, Lemma 2). If $T$ is translationally finite, then the small Euclidean motions in the definition of the tiling metric are necessarily translations. When $T$ is translationally finite, $\Omega_T$ is the space of tilings $T'$ such that every patch in $T'$ appears an in $T$, up to a small translation. 

If $T'\in\Omega_T$, then $T'+x\in\Omega_T$ for any $x\in\re^2$. This gives a continuous action of $\re^2$ on $\Omega_T$ and hence a dynamical system $(\Omega_T, \re^2)$. The orbit of $T$ is dense in $\Omega_T$ by definition. We shall say that $T$ is {\em minimal} whenever $(\Omega_T, \re^2)$ is minimal as a dynamical system (that is, every orbit is dense in $\Omega_T$). When $T$ is translationally finite, $T$ is minimal if and only if $T$ is {\em repetitive}, that is, for any patch $P\subset T$, there is an $R>0$ such that for any $x\in \re^2$ there is a translate of $P$ in $T$ whose support is contained in $B_R(x)$. 

A tiling $T$ is called {\em strongly aperiodic} if $\Omega_T$ contains no periodic tilings. It is a fact that if $T$ is repetitive and aperiodic, then it is strongly aperiodic (\cite{KP00}, Proposition 2.4).

Once and for all, we fix a tiling $\mathfrak{T}$ which is minimal, aperiodic, admits a finite prototile set $\p$ and which has finite local complexity. For notational convenience we will drop the subscript and simply denote $\Omega_\mathfrak{T} := \Omega$. Then $\Omega$ is compact, and the dynamical system $(\Omega,\re^2)$ is minimal. 

We will be primarily interested in the largest group of rotations which fixes $\om$. For translationally finite tilings, this group will be finite and cyclic -- for instance the largest group of rotations which fixes $\om$ for the Penrose tilings is the group generated by counterclockwise rotation by $\pi/5$, which is isomorphic to $\bz_{10}$. In contrast, Conway's pinwheel tilings have the property that within any given tiling there are tiles at infinite different orientations, and so the tiling space is fixed by the entire group $SO(\re, 2)\cong \mathbb{T}$. In either case, when $G$ is such a group, the map $G\times \om \to \om$ defined by $(g,T)\mapsto gT$ is jointly continuous and so defines a continuous left action of $G$ by homeomorphisms on $\om$, i.e. $\om$ is a left $G$-space. The space of orbits is denoted $\om/G = \{GT\mid T\in\om\}$ and is a compact Hausdorff space, see \cite{BDHS} and \cite{Wh10}. For $T\in \om$, we denote $G_T := \{g\in G\mid gT = T\}$; this is called the {\em stabilizer} subgroup for $T$. We have the following easy fact about the stabilizer subgroups.
\begin{lem}
For any $T\in\om$, $G_T$ is finite.
\end{lem}
\begin{proof}
Since the map $G\times \om \to \om$ defined by $(g,T)\mapsto gT$ is jointly continuous, $G_T$ is closed. Now, find $t\in T$ such that the origin is not in $t$. We can find an element $\theta\in \mathbb{T}$ such that $\theta t \cap t$ and $\theta t\setminus t$ both have nonempty interior. If $\theta \in G_T$, then $\theta t$ and $t$ are tiles in the same tiling $\theta T = T$, which is a contradiction. Thus $\theta\notin G_T$. Since the only infinite closed subgroup of $\mathbb{T}$ is $\mathbb{T}$, we are forced to conclude that $G_T$ is finite.
\end{proof}


We recall that if $A$ is a C*-algebra, $H$ is a locally compact group, and $\alpha: H \to$ Aut$(A)$ is a continuous homomorphism, then the {\em crossed product} is a C*-algebra denoted $A\rtimes_\alpha H$ which is meant to encode $A$, the group $G$, and the action $\alpha$. For instance, when $H$ is discrete and $A$ is unital, $A\rtimes_\alpha H$ contains a subalgebra isomorphic to $A$ as well as a unitary representation $\{u_g\mid g\in H\}\subset A\rtimes_\alpha G$ which implements the action $\alpha$ via the rule $\alpha(a) = u_gau_g^{*}$ for all $a\in A, g\in H$. When the action is understood, the subscript $\alpha$ is typically dropped. As mentioned in the introduction, if one thinks of a C*-algebra $A$ as a noncommutative topological space, then taking the crossed product by a group $H$ which is acting on $A$ can be thought of as the noncommutative analog of taking the space of $H$-orbits. This can be made concrete: if $H$ is a group which acts freely and properly on a space $X$, then the crossed product $C(X)\rtimes H$ is {\em strongly Morita equivalent} (a natural equivalence on C*-algebras) to $C(X/H)$, see \cite{Gre77}. For a reference on general C*-algebra crossed products, we refer the interested reader to \cite{Wil07}.

The compact space $\om$ leads to a C*-algebra $C(\om)$, the continuous complex-valued functions on $\om$ with pointwise product, sum, and supremum norm. Dynamical systems on $\om$ give rise to crossed products of $C(\om)$. Normally, the C*-algebra studied in the context of tilings is the crossed product $C(\om)\rtimes \re^2$. As mentioned in the introduction, our goal is to compute the K-theory of $C(\om)\rtimes (\re^2\rtimes G)$, which is isomorphic to a crossed product of the original tiling C*-algebra
\[
C(\om)\rtimes (\re^2\rtimes G) \cong (C(\om)\rtimes \re^2)\rtimes G,
\]
see \cite{Wil07}, Proposition 3.11 for the details.

\section{K-theory Computations}

K-theory is an invariant for C*-algebras -- given a C*-algebra $A$ one can produce an ordered pair of abelian groups $(K_0(A),K_1(A))$, this ordered pair is what we will refer to as the K-theory of $A$. K-theory is an invariant for strong Morita equivalence.

In previous study of $C(\om)\rtimes \re^2$, its K-theory has been computed using the following powerful theorem.

\begin{theo}{\em\bf (Connes-Thom)} Let $A$ be a C*-algebra and let $\alpha:\re\mapsto \text{Aut}(A)$ be a continuous homomorphism. Then
\[
K_i(A\rtimes_\alpha \re) \cong K_{i-1}(A)
\]
where the index is read modulo 2.
\end{theo}
Hence, computing the K-theory of $C(\om)\rtimes \re^2$ comes down to computing the K-theory of $C(\om)$. In the translationally finite case, this has been computed for both substitution tilings \cite{AP98} and projection method tilings \cite{Pu10} as the following:
\begin{eqnarray*}
K_0(C(\Omega)\rtimes \re^2) &\cong& H^0(\Omega, \bz) \oplus H^2(\Omega, \bz)\\
K_1(C(\Omega)\rtimes \re^2) &\cong& H^1(\Omega, \bz).
\end{eqnarray*}
where the $H^i(\om,\bz)$ are the \v Cech cohomology groups of $\om$ with integer coefficients. Computation of these cohomology groups has been a very active area of tilings research -- see \cite{AP98}, \cite{BDHS} and \cite{FHK}. It is thought that these groups give a measure of aperiodicity.

We use an analog of the Connes-Thom isomorphism to compute the K-theory of $C(\Omega)\rtimes (\re^2 \rtimes G)$ for rotation groups $G$. 

\begin{prop}
Let $A$ be a C*-algebra, let $G$ be a subgroup of rotations in $SO(\re,2)$ and let $\alpha: \re^2\rtimes G\to$ Aut$(A)$ be a continuous homomorphism. Then 
\[
K_i(A\rtimes_\alpha (\re^2\rtimes G)) \cong K_i(A\rtimes G)
\]
for $i = 0, 1$.
\end{prop}
\begin{proof}
See \cite{Kas88}, 5.10, \cite{CEN03}, \S7 or \cite{Ech08}, Theorem 3.1.
\end{proof}
One can see that one recovers the Connes-Thom isomorphism when $G$ is the trivial group. This says that we can calculate the $G$-equivariant K-theory of our C*-algebra by calculating the K-theory of $C(\om)\rtimes G$. A theorem of Green \cite{Gre77} states that if a group $H$ acts freely and properly on a space $X$, then the C*-algebras $C(X)\rtimes H$ and $C(X/H)$ are strongly Morita equivalent, and so have the same K-theory. For us, $G$ is compact and therefore its action on $\om$ is proper but in general it will not be free. 

We now arrive to the main theorem of the paper. It is essentially a recipe for computing the K-theory of the crossed products $(C(\om)\rtimes \re^2)\rtimes G$ under conditions which seem to be common for tilings of interest, and uses results of \cite{EE11}. In the proof, we restate the relevant results from \cite{EE11}, and then afterwards we present an example to show what is going on.

\begin{theo}\label{maintheorem} Let $\om$ be a tiling space with rotation group $G$. Suppose that
\begin{enumerate}
\item $F(\om) := \{ GT\in \om/G \mid gT = T \text{ for some } g\in G\setminus \{1\}\}$ is discrete in $\om/G$, and
\item $K_1(C(\om/G))$ is torsion-free.
\end{enumerate}
Let
\[
N_0 := \sum_{GT\in F(\om)} (|G_T|-1).
\]
Then
\begin{eqnarray*}
K_0((C(\om)\rtimes \re^2)\rtimes G) &\cong& K_0(C(\om/G)) \oplus \bz^{N_0}\\
K_1((C(\om)\rtimes \re^2)\rtimes G) &\cong& K_1(C(\om/G)).
\end{eqnarray*}
\end{theo}
\begin{proof}

Let $\rho$ denote the right regular representation of $G$, that is, $\rho:G \mapsto B(L^2(G))$ is determined by $\rho_g(\xi)(h) = \xi(hg)$ for all $\xi\in L^2(G)$. Each $\rho_g$ is necessarily unitary and $g$ is a homomorphism. Let
\[
A = \{f: \om\to \K(L^2(G)) \mid f(gT) = \rho_gf(T)\rho_g^*\text{ for all }T\in\om, g\in G\}
\]
where $\K(L^2(G))$ denotes the compact operators on $L^2(G)$. One can show that $C(\om)\rtimes G$ is isomorphic to $A$ (\cite{EE11}, Corollary 2.11). 


Given $T\in\om$, the stabilizer subgroup $G_T$ is a finite subgroup of $G$. For each character $\sigma\in \widehat G_T$ let
\[
L^2(G)_\sigma = \{\xi\in L^2(G)\mid \xi(gh) = \sigma^{-1}(h)\sigma(g)\text{ for all } g\in G, h\in G_T\}.
\]
Then we can decompose $L^2(G)$ into a finite direct sum of subspaces (\cite{EE11}, Example 3.2):
\[
L^2(G) = \bigoplus_{\sigma\in \widehat G_T} L^2(G)_\sigma
\] 
with each summand invariant under conjugation by elements of $\rho(G_T)$. Thus, we must have that for each $f\in A$ and $T\in\om$, 
\[
f(T) \in \bigoplus_{\sigma\in \widehat G_T} \K(L^2(G)_\sigma).
\]
Now, let 
\[
I = \left\{f \in A \mid f(T) \in \K\left(L^2(G)_{1_{G_T}}\right)\right\}
\] 
where $1_{G_T}$ denotes the character which takes value 1 on each element of $G_T$.
Then $I$ is an ideal of $A$ which is strongly Morita equivalent to $C(\om/G)$ (\cite{EE11}, Lemma 3.9). The quotient C*-algebra can be shown to be isomorphic to the direct sum of compact operators
\begin{equation}\label{blockdecomposition}
A/I \cong \bigoplus_{GT\in F(\om)}\bigoplus_{\sigma\in \widehat G_T\setminus \{1_{G_T}\}} \K(L^2(G)_\sigma).
\end{equation}
Since $G_T$ is a finite abelian group, $\widehat G_T$ has $|G_T|$ elements. Thus, $A/I$ is the direct sum of $N_0$ copies of compact operators (on possibly different Hilbert spaces), where $N_0$ is in the statement of the theorem. The short exact sequence of C*-algebras
\begin{equation}\label{shortexact}
0 \to I \to A \to A/I \to 0
\end{equation}
gives rise to the six-term exact sequence in K-theory
\begin{equation}\label{sixtermexact}
\xymatrix{
K_0(I) \ar[r] &K_0(A)\ar[r]& K_0(A/I)\ar[d]\\
K_1(A/I)\ar[u] &K_1(A)\ar[l]& K_1(I) \ar[l]}
\end{equation}
Since K-theory is invariant under isomorphism and Morita equivalence, respects direct sums, and $K_0(\K) = \bz$ and $K_1(\K) = 0$ for any algebra of compact operators $\K$, we obtain the exact sequence
\begin{equation}\label{longexact}
0 \to K_0(C(\om/G))\to K_0(C(\Omega)\rtimes G )\to \bz^{N_0}\stackrel{\partial}{\to} K_1(C(\om/G))\stackrel{\varphi}{\to} K_1(C(\Omega)\rtimes G ) \to 0.
\end{equation}

Whether $G$ is isomorphic to $\mathbb{T}$ or it is finite, the hypotheses of \cite{EE11}, Theorem 5.20 apply, and so the image of the map $\partial$ is a torsion subgroup of $K_1(C(\om/G))$. Since this group is torsion-free by assumption, we have that $\partial$ is the zero map, and so $\varphi$ is an isomorphism. Furthermore, the left part of (\ref{longexact}) is a short exact sequence which splits since $\bz^{N_0}$ is a free abelian group.
\end{proof}

\begin{rmk}\label{postthmrmk}
We make two remarks about the above theorem.
\begin{enumerate}
\item For substitution tilings, the situation simplifies a little bit. In this case, the first assumption of Theorem \ref{maintheorem} is always satisfied -- in this case elements of $F(\om)$ will correspond to rotationally-symmetric patches $P$ such that $P = P(0)$ and for which the sequence $(\omega^n(P)(0))_{n\in\NN}$ is periodic. In addition, for the translationally finite case, the space $\om/G$ is the inverse limit of a reduced Anderson-Putnam complex and so the K-theory of $C(\om/G)$ is computable -- $K_0$ will be the direct sum of the even cohomology groups of $\om/G$ while $K_1$ will be the direct sum of the odd cohomology groups. Even in some non-translationally finite cases (such as the pinwheel tiling), the space $\om/G$ is also the inverse limit of 2-dimensional cell complexes (\cite{BDHS}, \cite{FWW}), and so we can likewise calculate the K-theory of $C(\om/G)$. Calculations of this type are excellently explained in \cite{AP98}, \cite{BDHS}, and \cite{ORS}
\item We know of no examples of tilings for which $K_1(C(\om/G))$ contains a torsion subgroup, but we do not know if it is torsion free in general. In fact, for every example we have computed, this group is either $0$ or $\bz$, see Section 4. 
\end{enumerate}
\end{rmk}

\begin{ex} Penrose tiling, $G=\bz_{10}$.\label{penrosez10ex}

The Penrose tiling is a well-studied tiling with finite rotational symmetry. If we let $r$ denote counterclockwise rotation by $\pi/5$, we have that $\om$ is acted upon by $\langle r \rangle\cong\bz_{10}$. The space $\om/\bz_{10}$ has only two points with trivial stabilizers, each fixed under $r^2$; call them $T_1$ and $T_2$. We note that while $\om$ is fixed by the rotation $r$, no single tiling in $\om$ is.

We have that $\mathcal{K}(L^2(\bz_{10})) \cong \mathbb{M}_{10}$, and relative to a suitable basis $\rho_{r^2}$ is the diagonal $10\times 10$ matrix with entries $(1, 1, z^2, z^2, z^4, z^4,z^6, z^6, z^8, z^8)$, where $z = e^{\frac{i\pi}{5}}$.
We have
\[
C(\Omega)\rtimes \bz_{10} \cong \{ f\in C(\Omega, \mathbb{M}_{10})\mid f(gT) = \rho_gf(T)\rho_g^*\text{ for all }T\in\om, g\in \bz_{10}\}.
\]
In particular, for $i = 1, 2$ we have that $f(T_i)\rho_{r^2} = \rho_{r^2}f(T_i)$, and so each $f(T_i)$ must take the form
\[
f(T_i) = \left[\begin{array}{ccccc}
     B_1&&&&\\
     &B_2&&&\\
     &&B_3&&\\
     &&&B_4&\\
     &&&&B_5
\end{array}\right], \hspace{1cm} B_j\in \mathbb M_2.
\]
Relative to the basis $(v_1, \dots, v_{10})$ which gives $\rho_{r^2}$ the form above, we have
\[
L^2(\bz_{10})_{1_{G_{T_i}}} = \textnormal{span}\{v_1, v_2\}.
\]
Hence our $I$ is all the functions in $C(\Omega)\rtimes \bz_{10}$ such that $f(T_i)$, $i=1,2$, is of the form
\[
f(T_i) = \left[\begin{array}{ccccc}
     B_1&&&&\\
     &0_2&&&\\
     &&0_2&&\\
     &&&0_2&\\
     &&&&0_2
\end{array}\right], \hspace{1cm} B_1\in \mathbb M_2.
\]
where $0_2$ denotes the $2\times 2$ zero matrix. Let $q\in\mathbb{M}_{10}$ be the projection
\[
q = \left[\begin{array}{cc}
     0_2&\\
     &I_8
\end{array}\right]
\]
where $I_8$ is the $8\times 8$ identity matrix. The $*$-homomorphism 
\[
\psi: C(\Omega)\rtimes \bz_{10} \to \left(\bigoplus_{i=1}^{4}\mathbb{M}_2\right)\oplus \left(\bigoplus_{i=1}^{4}\mathbb{M}_2\right) =:Q
\]
\[
f \mapsto (qf(T_1), qf(T_2))
\]
has kernel $I$, and so we obtain the short exact sequence

$$
0 \longrightarrow I \longrightarrow  C(\Omega)\rtimes \bz_{10}  \stackrel{\psi}{\longrightarrow} Q \longrightarrow 0.
$$
This short exact sequence leads to six-term exact sequence in K-theory,
$$
\xymatrix{
K^0(I) \ar[r] &K_0(C(\Omega)\rtimes \bz_{10} )\ar[r]& K_0(Q)\ar[d]\\
K_1(Q)\ar[u] &K_1(C(\Omega)\rtimes \bz_{10} )\ar[l]& K^1(I) \ar[l]}$$
Now, as in the proof of Theorem \ref{maintheorem} $I$ is strongly Morita equivalent to $C(\Omega/\bz_{10})$. As mentioned in Remark \ref{postthmrmk}, $\om/\bz_{10}$ is the inverse limit of two dimensional cell complexes, and so $K_0(C(\om/\bz_{10}))$ is isomorphic to the direct sum of the even cohomology groups of $\om/\bz_{10}$ while $K_1(C(\om/\bz_{10}))$ is the direct sum of its odd cohomology groups. Using standard methods, one finds that $H^0(\om/\bz_{10},\bz)\cong \bz$, $H^1(\om/\bz_{10},\bz)\cong \bz$, and $H^2(\om/\bz_{10},\bz)\cong \bz^2$ (see \cite{AP98} for a detailed prescription for calculating the cohomology of inverse limits of cell complexes, \cite{StThesis} Section 5.2 for the case of orbit spaces, or \cite{ORS} where this cohomology was calculated for the Penrose tiling).

One also has that $K_0(\mathbb{M}_n) = \bz$, and $K_1(\mathbb{M}_n) =0$, and so we obtain 
\[
\xymatrix{
0 \ar[r] & \bz^3 \ar[r] & K_0(C(\Omega)\rtimes \bz_{10}) \ar[r] & \bz^8 \ar[r]^\partial & \bz \ar[r] & K_1(C(\Omega)\rtimes \bz_{10} )\ar[r] & 0
}
\]
Since the image of $\partial$ must have torsion, $\partial$ must be zero and so
\[
K_0(C(\Omega)\rtimes \bz_{10}) \cong \bz^{11}
\]
\[
K_1(C(\Omega)\rtimes \bz_{10}) \cong \bz
\]
\end{ex}

\section{Table of K-theory for Examples}
We give a table for the K-theory of the crossed products for some well-known tilings complete with references of where they appear in the literature. Most also appear in \cite{GS86}. In the following, we shorten $(C(\om)\rtimes \re^2)\rtimes G$ to $A$.

\noindent\begin{tabular}{|l||c|c|c||c|}
\hline
{\bf Tiling} & $G$& $K_0(A)$ & $K_1(A)$& Reference\\
\hline\hline
T\"{u}bingen Triangle & $\bz_{10}$&$\bz^{5}$ & $0$& \cite{BKSZ90}\\
\hline
Penrose & $\bz_{10}$&$\bz^{11}$ & $\bz$& \cite{ORS}, \cite{BDHS}, Example 11\\
\hline
Octagonal & $\bz_8$&$\bz^{11}$ & $\bz$ & \cite{StThesis}, Example 6.4.5\\
\hline
Half-hex & $\bz_6$ &   $\bz\left[\frac14\right]\oplus \bz^4$& 0& \cite{Rand06}, Example 7.2\\
\hline
Chair & $\bz_4$ & $\bz\left[\frac14\right]\oplus \bz^4$& 0& \cite{BDHS}, Example 10\\
\hline
Ammann A2 & $\bz_2$ & $\bz^5 \oplus \bz_2^2$& $\bz$&\cite{AP98}, Example 10.3 \\
\hline
Pinwheel & $\mathbb{T}$ & $\bz\left[\frac{1}{25}\right]\oplus \bz\left[\frac13\right]^2\oplus\bz^{12} \oplus \bz_2$ & $\bz$&\cite{BDHS}, Section 7, \cite{FWW}\\
\hline
\end{tabular}

\vspace{0.5cm}
\begin{rmk}
The C*-algebra most studied in the literature is not exactly $C(\om)\rtimes \re^2$, rather that of the groupoid derived by reducing the $\re^2$ action to a canonical transversal $\Xi\subset \om$. This algebra is typically denoted $A_{\mathfrak{T}}$, and has been studied by many authors, both in the translationally finite case where it is strongly Morita equivalent to $C(\om)\rtimes \re^2$ (\cite{AP98}, \cite{KelNCG}, \cite{KP00}, \cite{Pu00}) and in the infinite rotation case (\cite{Wh10}). In the case of translationally finite tilings arising from primitive substitutions, we studied the crossed products $A_\mathfrak{T}\rtimes G$ in \cite{St12} (based on the author's PhD thesis \cite{StThesis}). In this case, it is straightforward to show that $A_\mathfrak{T}\rtimes G$ is strongly Morita equivalent to $C(\om)\rtimes (\re^2 \rtimes G)$, and so has the same K-theory (see \cite{StThesis}, Theorem 5.5.2 for the details).

In \cite{St12}, we used techniques of Putnam and Phillips to show that $A_\mathfrak{T}\rtimes G$ is simple, has real rank zero, stable rank one, has a unique trace, and that the order on $K_0$ is determined by this trace. This is (incomplete) progress towards showing $A_\mathfrak{T}\rtimes G$ is classifiable by its K-theory. If one could show that $A_\mathfrak{T}\rtimes G$ is classifiable in this case, then one could conclude that, for instance, $A_\mathfrak{T}\rtimes G$ is an {\em approximately finite} or {\em AF} C*-algebra when $\mathfrak{T}$ is the half-hex tiling. Interestingly, $A_\mathfrak{T}\rtimes G$ has a canonical AF-subalgebra, denoted in \cite{St12} as $A_\omega\rtimes G$, and in the case of the half-hex, $K_0(A_\omega\rtimes G) = \bz\left[\frac14\right]$. Hence even if $A_\mathfrak{T}\rtimes G$ is proven to be classifiable and hence AF, it is not isomorphic to the obvious candidate.
\end{rmk}
{\bf Acknowledgements.} The author wishes to thank Thierry Giordano, Ian Putnam, Michael Whittaker, Daniel Gon\c {c}alves, David Handelman, Heath Emerson and Siegfried Echterhoff for many helpful conversations and suggestions regarding this work. In particular, great gratitude is due Thierry Giordano who supervised work on the author's PhD thesis from which this paper is derived.
\bibliography{C:/Users/Charles/Dropbox/Research/bibtex}{}
\bibliographystyle{plain}
\end{document}